\documentclass{amsart}

\ifx\MyBeamer\undefined
\usepackage[breaklinks]{hyperref}
\usepackage{fullpage}
\else
\if\MyBeamer t
\usepackage[breaklinks]{hyperref}
\usepackage{beamerarticle}
\usepackage{fullpage}
\fi\fi

\usepackage{amsmath}
\usepackage{amssymb}
\usepackage{mathrsfs}
\usepackage{amsthm}
\usepackage{aliascnt}
\usepackage[utf8]{inputenc}
\usepackage[T1]{fontenc}
\usepackage[nofancy]{svninfo}
\ifx\FrenchText\undefined
\usepackage[british]{babel}
\else
\usepackage[british,french]{babel}
\AtBeginDocument{%
\catcode`\:=12%
\catcode`\!=12%
}
\fi

\makeatletter

\ifx\MyBeamer\undefined

\def\equationautorefname~#1\null{(#1)}
\fi

\ifx\MyBeamer\undefined

\ifx\thmnum\undefined
\newtheorem{thmnum}{}[section]
\fi

\newcommand{\mynewthm}[3][]{%
  \newaliascnt{#2}{thmnum}%
  \newtheorem{#2}[#2]{#3}%
  \aliascntresetthe{#2}%
  \newtheorem*{#2*}{#3}%
  \expandafter\newcommand\csname #2autorefname\endcsname{#3}%
  \expandafter\renewcommand\csname the#2\endcsname{\thethmnum}%
}

\newtheorem*{clm}{Claim}
\newenvironment{clmprf}{%
  \begin{proof}[Proof of claim]%
  }{\end{proof}}

\else

\let\xxx=\frametitle
\def\frametitle#1{%
  \xxx{%
    \setbeamercolor*{math text}{use={titlelike,my math text},fg=titlelike.fg!80!my math text.fg}%
    #1}%
  \setbeamercolor{math text}{use=my math text,fg=my math text.fg}%
}

\newcommand{\beamerenv}[3]{%
\newenvironment<>{#1}%
{%
  \setbeamercolor{temp}{structure}%
  \setbeamercolor{structure}{fg=#2}%
  \setbeamercolor{block body}{use=structure,bg=structure.fg!5!white}%
  \begin{#3}%
}%
{\end{#3}\setbeamercolor{structure}{temp}}}

\newcommand{\mynewthm}[3][green!50!black]{%
  \newtheorem*{#2x}{#3}%
  \beamerenv{#2}{#1}{#2x}%
}

\beamerenv{other}{violet}{block}

\fi

\ifx\FrenchText\undefined
\newcommand{\myiffrench}[2]{#2}
\else
\newcommand{\myiffrench}[2]{\iflanguage{french}{#1}{#2}}
\fi

\theoremstyle{plain}
\mynewthm[red]{thm}{\myiffrench{Théorème}{Theorem}}
\mynewthm{prp}{Proposition}
\mynewthm[purple]{lem}{\myiffrench{Lemme}{Lemma}}
\mynewthm[orange]{fct}{\myiffrench{Fait}{Fact}}
\mynewthm[purple!50!white]{cor}{\myiffrench{Corollaire}{Corollary}}

\theoremstyle{definition}
\mynewthm[green!80!black]{dfn}{\myiffrench{Définition}{Definition}}
\mynewthm{hyp}{\myiffrench{Hypothèse}{Hypothesis}}
\mynewthm{conv}{Convention}
\mynewthm{conj}{Conjecture}
\mynewthm{ntn}{Notation}
\mynewthm{cst}{Construction}

\theoremstyle{remark}
\mynewthm[yellow!90!black]{rmk}{\myiffrench{Remarque}{Remark}}
\mynewthm{qst}{Question}
\mynewthm{exc}{\myiffrench{Exercice}{Exercise}}
\mynewthm{exm}{\myiffrench{Exemple}{Example}}

\ifx\MyBeamer\undefined

\newcommand{\myenumlabel}[1]{\textnormal{(\roman{#1})}}

\fi

\newcounter{cycprfcnt}
\newcommand{\cycprfpreamble}%
{%
  \setcounter{cycprfcnt}{1}
  \setlength{\itemindent}{0.5\leftmargin}%
  \setlength{\leftmargin}{0pt}%
  \newcommand{\cpcurr}{\myenumlabel{cycprfcnt}}%
  \newcommand{\cpnext}{\addtocounter{cycprfcnt}{1}\cpcurr}%
  \newcommand{\impnext}{\cpcurr{} $\Longrightarrow$ \cpnext.}%
}%

{\begin{list}{\impnext}%
  {\cycprfpreamble}}%
{\qedhere\end{list}}%

\newenvironment{cycprf*}%
{\begin{list}{\impnext}%
  {\cycprfpreamble}}%
{\end{list}}%

\def\indsym#1#2{%
  \setbox0=\hbox{$\m@th#1x$}%
  \kern\wd0%
  \hbox to 0pt{\hss$\m@th#1\mid$\hbox to 0pt{$\m@th#1^{#2}$\hss}\hss}%
  \lower.9\ht0\hbox to 0pt{\hss$\m@th#1\smile$\hss}%
  \kern\wd0}

\def\nindsym#1#2{%
  \setbox0=\hbox{$\m@th#1x$}%
  \kern\wd0%
  \hbox to 0pt{\hss$\m@th#1\not$\kern1.4\wd0\hss}
  \hbox to 0pt{\hss$\m@th#1\mid$\hbox to 0pt{$\m@th#1^{#2}$\hss}\hss}%
  \lower.9\ht0\hbox to 0pt{\hss$\m@th#1\smile$\hss}%
  \kern\wd0}

\def\dotminussym#1#2{%
  \setbox0=\hbox{$\m@th#1-$}%
  \kern.5\wd0%
  \hbox to 0pt{\hss\hbox{$\m@th#1-$}\hss}%
  \raise.6\ht0\hbox to 0pt{\hss$\m@th#1.$\hss}%
  \kern.5\wd0}
\newcommand{\dotminus}{\mathbin{\mathpalette\dotminussym{}}}

\renewcommand{\emptyset}{\varnothing}
\renewcommand{\setminus}{\smallsetminus}

\newcommand{\sfrac}[2]{\hbox{$\frac{#1}{#2}$}}
\newcommand{\half}[1][1]{\sfrac{#1}{2}}

\DeclareMathOperator{\tp}{tp}

\DeclareMathOperator{\Th}{Th}

\DeclareMathOperator{\tS}{S}

\newcommand{\fS}{\mathfrak{S}}

\newcommand{\cC}{\mathcal{C}}

\newcommand{\cL}{\mathcal{L}}

\newcommand{\cM}{\mathcal{M}}

\newcommand{\bA}{\mathbf{A}}

\newcommand{\bC}{\mathbf{C}}

\newcommand{\bN}{\mathbf{N}}
\newcommand{\bP}{\mathbf{P}}

\newcommand{\bR}{\mathbf{R}}

\newcommand{\bZ}{\mathbf{Z}}
\newcommand{\ba}{\mathbf{a}}
\newcommand{\bb}{\mathbf{b}}
\newcommand{\bc}{\mathbf{c}}

\makeatother

\DeclareMathOperator{\Sph}{Sph}
\DeclareMathOperator{\rad}{rad}
\DeclareMathOperator{\Sq}{Sq}

\newcommand{\llangle}{\langle\!\!\langle}
\newcommand{\rrangle}{\rangle\!\!\rangle}
\newcommand{\<}{\llangle}
\renewcommand{\>}{\rrangle}

\begin{document}

\title{Model theoretic properties of metric valued fields}

\author{Itaï \textsc{Ben Yaacov}}

\address{Itaï \textsc{Ben Yaacov} \\
  Université Claude Bernard -- Lyon 1 \\
  Institut Camille Jordan, CNRS UMR 5208 \\
  43 boulevard du 11 novembre 1918 \\
  69622 Villeurbanne Cedex \\
  France}

\urladdr{\url{http://math.univ-lyon1.fr/~begnac/}}

\thanks{Author supported by
  ANR chaire d'excellence junior THEMODMET (ANR-06-CEXC-007) and
  by the Institut Universitaire de France.}

\thanks{The author would like to thank Ehud Hrushovski and
  C.\ Ward Henson for several inspiring discussions.}

\svnInfo $Id: MVF.tex 1572 2013-05-07 08:41:38Z begnac $
\thanks{\textit{Revision} {\svnInfoRevision} \textit{of} \today}


\begin{abstract}
  We study model theoretic properties of valued fields
  (equipped with a real-valued multiplicative valuation),
  viewed as metric structures in continuous first order logic.

  For technical reasons we prefer to consider not the valued field
  $(K,|{\cdot}|)$ directly, but rather the associated projective
  spaces $K\bP^n$, as bounded metric structures.

  We show that the class of (projective spaces over)
  metric valued fields is elementary, with theory $MVF$,
  and that the projective spaces
  $\bP^n$ and $\bP^m$ are biïnterpretable for every
  $n,m \geq 1$.
  The theory $MVF$ admits a model completion
  $ACMVF$, the theory of algebraically closed metric valued fields
  (with a non trivial valuation).
  This theory is strictly stable (even up to perturbation).

  Similarly, we show that the theory of real closed metric valued
  fields, $RCMVF$, is the model companion of the theory of formally
  real metric valued fields, and that it is dependent.
\end{abstract}

\maketitle

\section{The theory of metric valued fields}

Let us recall some terminology from Berkovich \cite{Berkovich:SpectralTheory}.
A \emph{semi-normed ring} is a unital commutative ring $R$
equipped with a mapping $|{\cdot}|\colon R \to \bR^{\geq 0}$
such that
\begin{enumerate}
\item $|1| = 1$,
\item $|xy| \leq  |x| |y|$,
\item $|x + y| \leq |x| + |y|$.
\end{enumerate}
If $|x| = 0 \Longrightarrow x = 0$ then $|{\cdot}|$ is a \emph{norm}.
A semi-norm is \emph{multiplicative} if $|xy| = |x||y|$.
A multiplicative norm is also called a \emph{valuation}.
Thus, a valued field is equipped with a natural metric structure
$d(x,y) = |x-y|$.
In some contexts, a valuation is allowed to take values in
$\Gamma \cup \{0\}$ where $(\Gamma,\cdot)$ is an arbitrary ordered
Abelian group and $0 < \Gamma$, but this will not be the case in the
present text.
When we wish to make this explicit we shall refer to our fields as
\emph{metric valued fields}.

If $K$ is a complete valued field then either
$K \in \{\bR,\bC\}$ and $|{\cdot}|$ is the usual absolute value to
some power (in which case $|{\cdot}|$ is \emph{Archimedean})
or $|x+y| \leq |x| \vee |y|$
($|{\cdot}|$ is \emph{non Archimedean}, or \emph{ultra-metric}).
From a model theoretic point of view,
Archimedean valued fields, being locally compact,
resemble finite structures of classical logic and are thus far less
interesting than their ultra-metric counterparts.
On the other hand, while everything we do here applies to arbitrary
valued fields, including Archimedean ones, restricting our attention
to the ultra-metric case does allow us many simplifications.
Thus, with very little loss of generality, we shall only consider
ultra-metric valued fields.

\begin{conv}
  Throughout, unless explicitly stated otherwise, by a \emph{valued field} we mean a non Archimedean one.
\end{conv}

The valuation is said to be \emph{trivial} if $|x| = 1$ for every $x \neq 0$.
It is \emph{discrete} if the image of $|{\cdot}|$ on $K^\times$ is discrete.
Clearly every trivial valuation is discrete.
On the other hand, a non trivial valuation on an algebraically (or separably) closed field cannot be discrete.

A non trivially valued field is unbounded as a metric space, and therefore does not fit in the framework of standard bounded continuous logic.
One device we use quite often with Banach space structures (Banach spaces, Banach lattices, and so on) is to restrict our attention to the structure formed by the closed unit ball.
This approach may seem natural for valued fields as well, since the unit ball is simply the corresponding valuation ring.
However, in the case of a non discrete valuation this approach is not adequate, as shown by the following result.

\begin{prp}
  Let $(K,|{\cdot}|)$ be a field equipped with a non discrete valuation, and let $R = (R,0,1,-,+,\cdot,|{\cdot}|)$ be its valuation ring.
  Then $R$ cannot be saturated as a metric structure (i.e., in the sense of continuous logic).
  In fact, it cannot even realise every type over $\emptyset$.
\end{prp}
\begin{proof}
  Since $R$ is not discrete we can find for each $n$ an element $a_n \in R$ such that $1-2^{-n} < |a_n| < 1$.
  Such an element is not invertible in $R$, and worse, for every $b \in R$ we have $|a_nb| < |b| \leq 1$, whereby $|a_nb-1| = 1$.
  In other words, each $a_n$ satisfies the assertion that $\inf_y \, |xy-1| = 1$.
  Thus in an ultra-power of $R$ there exists $a$ such that $|a| = \inf_y \, |ay-1| = 1$.
  Since every element of $R$ of value $1$ is invertible, such an element cannot exist in $R$.
\end{proof}

Therefore, if we are to hope for a reasonable model theoretic treatment of valued fields, the entire field should be considered as an unbounded structure.
Unbounded metric structures are discussed in \cite{BenYaacov:Unbounded}, where we also introduce an \emph{emboundment} process whereby unbounded structures can be turned into bounded ones through the addition of a single point at infinity.
In the case of a valued field, the resulting structure can be naturally identified (as a set of points) with the projective line, which is a natural object in itself.
For our purposes it will be more convenient to consider the projective line directly, rather than as the emboundment of the field (and one can check that the two structures are interdefinable).
As in the general case of emboundment, even though the field language contains function symbols, these do not pass on to the projective line.
Indeed, the addition map $\bigl( [x:1], [y:1] \bigr) \mapsto [x+y:1]$ is ill defined at $\bigl( [1:0], [1:0] \bigr)$, and similarly $\bigl( [x:1], [y:1] \bigr) \mapsto [xy:1]$ is ill defined at $\bigl( [0:1], [1:0] \bigr)$.
We shall therefore have to do, at least for the time being, with a purely relational language (this will be remedied later on when we consider projective spaces of higher dimension).

We recall that the projective $n$-space over a field $K$ is the quotient $(K^{n+1} \setminus \{0\})/K^\times$.
The class of $(a) = (a_i) = (a_0,\ldots,a_n)$ is denoted $\ba = [a] = [a_i] = [a_0:\ldots:a_n]$.
Dividing by a coordinate with maximal value we see that any member of $K\bP^n$ can be written as $[a_i]$ where $\bigvee |a_i| = 1$.
From now on we shall assume that all the representatives are of this form, which determines them up to a multiplicative factor from the group $\{x \in K\colon |x| = 1\} = \ker |{\cdot}|$.

\begin{ntn}
  Let $\bar X = (X_0,\ldots,X_{n-1})$ denote $n$ formal unknowns.
  We let $\bar X^*$ denote a copy of $\bar X$, and let
  $\bZ^h[\bar X] \subseteq \bZ[\bar X,\bar X^*]$
  denote the ring of polynomials in
  $\bar X$, $\bar X^*$ which are homogeneous
  in each pair $(X_i,X^*_i)$ separately
  (which is stronger than being homogeneous in all the variables
  simultaneously).
  For a polynomial $Q(\bar X,\bar X^*) \in \bZ^h[\bar X]$
  let $\bar Q(\bar X) = Q(\bar X,\bar 1) \in \bZ[\bar X]$.

  For $P(\bar X) \in \bZ[\bar X]$ let
  $\deg_{\bar X} P = (\deg_{X_0} P,\ldots,\deg_{X_{n-1}} P) \in \bN^n$
  and let
  $P^*(\bar X^*)
  = (\bar X^*)^{\deg_{\bar X} P}
  = \prod (X_i^*)^{\deg_{X_i} P}
  \in \bZ[\bar X^*]$,
  $P^h(\bar X,\bar X^*) = P(\frac{\bar X}{\bar X^*}) P^*(\bar X^*)$.
  Then $P^h \in \bZ[\bar X,\bar X^*]$
  is unique such that $P = \overline{P^h}$
  and no $X_i^*$ can be factored out of $P^h$.
  We call $P^h$ the \emph{homogenisation} of $P$ and observe that $P \mapsto P^h$ is multiplicative.
  Conversely, every $Q \in \bZ^h[\bar X]$
  can be written uniquely as
  $\bar Q^h\cdot(\bar X^*)^{\alpha(Q)}$,
  where $\alpha(\bar Q) \in \bN^n$ is a multi-exponent.
\end{ntn}

We now have everything we need to define the language and theory of (projective lines of) metric valued fields in ordinary (i.e., bounded) continuous logic, as presented in \cite{BenYaacov-Usvyatsov:CFO} or \cite{BenYaacov-Berenstein-Henson-Usvyatsov:NewtonMS}.

\begin{dfn}
  We define the language $\cL_{\bP^1}$ to consist of a constant symbol $\infty$ and one $n$-ary, $[0,1]$-valued predicate symbol $\|P(\bar x)\|$ for each $n$ and each polynomial $P \in \bZ[X_0,\ldots, X_{n-1}]$.
  (There is some abuse of notation here, since $P$ does not determine $n$ but this will not cause any problems.)
\end{dfn}

\begin{dfn}
  \label{dfn:KP1}
  For a valued field $(K,|{\cdot}|)$, we view $K\bP^1$ as an $\cL_{\bP^1}$-pre-structure by:
  \begin{gather*}
    \infty := [1:0],
    \qquad
    \|P(\bar \ba)\| := |P^h(\bar a,\bar a^*)|,
    \qquad
    d(\ba,\bb) := \|\ba-\bb\| = |ab^*-a^*b|.
  \end{gather*}
  This is independent of the choice of representatives, keeping mind that we only consider representatives for $[a:a^*] \in K\bP^1$ such that $|a| \vee |a^*| = 1$.
\end{dfn}

We observe that $|a^*| = \|\ba-\infty\| = d(\ba,\infty)$, and we shall use $\|x^*\|$ as an abbreviation for the formula $d(x,\infty)$.
For $P(\bar X) \in Z[\bar X]$ we have $|P^*(\bar a^*)| = \prod |a_i^*|^{\deg_{X_i} P}$, and we shall similarly use $\| P^*(\bar x) \|$ as an abbreviation for $\prod \|x_i^*\|^{\deg_{X_i} P}$.
We notice that $\|P(\bar \ba)\| = |P(\bar \ba)| \| P^*(\bar \ba) \|$ (if $\ba_i \in K \subseteq K\bP^1$ whenever $\deg_{X_i} P > 0$ then this
makes sense, and otherwise $\|P^*(\bar \ba)\| = 0$, and the identity still makes sense).

\begin{dfn}
  We define $MVF$, the theory of projective lines over
  metric valued fields, to consist of the following axioms.
  In axiom \autoref{ax:Perm}, $\sigma \in \fS_n$ is a permutation and $(X_0,\ldots,X_{n-1})^\sigma = (X_{\sigma 0},\ldots,X_{\sigma(n-1)})$.
  \begin{align*}
    \label{ax:Norm}
    \tag{Norm}
    & \|x\| \vee \|x^*\| = 1
    \\
    \label{ax:Perm}
    \tag{Perm}
    &
    \|P(\bar x)\| = \|Q(\bar x^\sigma,\bar y)\|
    &&
    \left( P(\bar X) = Q(\bar X^\sigma,\bar Y) \right)
    \\
    \label{ax:Ult}
    \tag{Ult}
    & \|\bar x^*\|^\alpha \|P(\bar x)\|
    \leq
    \|\bar x^*\|^\beta \|Q(\bar x)\|
    \vee
    \|\bar x^*\|^\gamma \|R(\bar x)\|
    && \bigl(
    (\bar X^*)^\alpha P^h = (\bar X^*)^\beta Q^h - (\bar X^*)^\gamma R^h
    \bigr)
    \\
    \label{ax:Prod}
    \tag{Prod}
    & \|(PQ)(\bar x)\| = \|P(\bar x)\|\|Q(\bar x)\|
    \\
    \tag{Dist}
    & d(x,y) = \|x-y\|
    \\
    \label{ax:Lin}
    \tag{Lin}
    & \exists y \,
    \|P(\bar x,y)\| = 0
    && \bigl( \deg_Y P(\bar X,Y) = 1 \bigr)
  \end{align*}
\end{dfn}

Axioms are universally quantified, so axiom \autoref{ax:Norm}, for example, should be understood as the sentence $\sup_x \, \bigl| 1 - \|x\| \vee \|x^*\| \bigr|$ (where we recall the convention of continuous logic, that zero is ``True''), and similarly for the other axioms which appear quantifier-free.
In the last axiom, the existential quantifier should be understood in the approximate sense: there exists $y$ such that $\|P(\bar x,y)\|$ is as close as desired to zero, or formally, $\sup_{\bar x} \, \inf_y \, \|P(\bar x,y)\|$.
In continuous logic one simply cannot express directly the existence of some $y$ such that something holds precisely (e.g., such that $\|P(\bar x,y)\|$ is precisely zero), although in concrete situations one can prove that approximate existence implies precise existence, as is the case with \autoref{lem:UniqueQuotient} below.

It follows immediately from the axioms that $\|P\| = \|{-P}\|$ and $\|\infty\| = 1$.

\begin{lem}
  \label{lem:UniqueQuotient}
  Assume that $\cM \vDash MVF$.
  Then for every $P,Q \in \bZ[\bar X]$ and every $\bar a \in M^n$, if $\|P(\bar a)\|  \|Q^*(\bar a)\| \neq 0$ then there exists a unique $b \in M$ such that $\|P(\bar a)b - Q(\bar a)\| = 0$, i.e., $\|R(\bar a,b)\| = 0$ where $R = PY-Q$.
  Moreover, this $b$ is distinct from $\infty$.
\end{lem}
\begin{proof}
  Let $\alpha = \deg_{\bar X} Q \dotminus \deg_{\bar X} P$, $\beta = \deg_{\bar X} P \dotminus \deg_{\bar X} Q$, so $R^h = (\bar X^*)^\alpha P^h Y - (\bar X^*)^\beta Y^* Q^h$.
  Then
  \begin{gather*}
    (\bar X^*)^\alpha \bigl[ (Y-Z) P\bigr]^h
    =
    Z^* R(\bar X,Y)^h - Y^* R(\bar X,Z)^h.
  \end{gather*}
  By the ultra-metric axiom \autoref{ax:Ult}:
  \begin{align*}
    d(y,z)
    &
    \leq
    \frac{\|R(\bar a,y)\|\|z^*\| \vee \|R(\bar a,z)\|\|y^*\|}
    {\|P(\bar a)\|\|\bar a^*\|^\alpha}
    \leq
    \frac{\|R(\bar a,y)\| \vee \|R(\bar a,z)\|}
    {\|P(\bar a)\|\|Q^*(\bar a)\|}.
  \end{align*}
  Uniqueness follows.
  By the linear solution axiom \autoref{ax:Lin} there exists a sequence $(b_n)$ such that $\|P(\bar a)b_n - Q(\bar a)\| \to 0$.
  It follows from our argument above that this is a Cauchy sequence, and its limit $b$ is a solution.
  Finally, $\|R(\bar a,\infty)\| = \|P(\bar a)\|\|\bar a^*\|^\alpha \neq 0$, so $b \neq \infty$.
\end{proof}

When $b$ is as in the lemma we write $b = \frac{Q(\bar a)}{P(\bar a)}$, and if $P = 1$ we write $b = Q(\bar a)$.

\begin{thm}
  An $\cL_{\bP^1}$-structure is a model of $MVF$ if and only if it is isomorphic to $K\bP^1$ for some complete valued field $K$.
\end{thm}
\begin{proof}
  Only one direction requires a proof.
  Assume therefore that $\cM \vDash MVF$.
  Let $K = M \setminus \{\infty\}$.
  For $a,b \in K$, and with the notation above, $a+b = \frac{a+b}{1}$ is the unique solution for $\|Y - a-b\| = 0$.
  We may similarly define $ab$, $-a$, as well as the constants $0$ and $1$, and since $\|a^*\| \neq 0$ we may also define $|a| = \frac{\|a\|}{\|a^*\|}$.

  Let us check that $(K,0,1,-,+,\cdot,|{\cdot}|)$ is a valued field.
  For this purpose, we shall use brackets to enclose expressions involving the field operations of $K$, whereas expressions outside brackets correspond to polynomials over $\bZ$.
  Axiom \autoref{ax:Perm} ensures that we need not worry about the order of variables in a polynomial nor about dummy variables, and will be used implicitly throughout.

  In order to see that addition is associative, for example, observe that
  \begin{gather*}
    X^*(W - Y - Z - T)^h  = Y^*Z^*(W - X - T)^h + W^*T^*(X - Y - Z)^h.
  \end{gather*}
  Then by \autoref{ax:Ult} and the fact that $\|P\| = \|-P\|$, that for all $a,b,c \in K$,
  \begin{gather*}
    \bigl\| [a+b]^* \bigr\| \bigl\| [(a+b)+c] - a - b - c \bigr\| = 0
    \qquad \Longrightarrow \qquad
    \bigl\| [(a+b)+c] - a - b - c \bigr\| = 0.
  \end{gather*}
  A similar argument yields $\bigl\| [a+(b+c)] - a - b - c \bigr\| = 0$.
  It follows from the uniqueness clause of \autoref{lem:UniqueQuotient} that $[(a+b) + c] = [a + (b + c)]$.
  Similarly,
  \begin{gather*}
    X^*(W - YZT)^h  = Y^*Z^*(W - XT)^h + W^*(XT - YZT)^h.
  \end{gather*}
  Using also axiom \autoref{ax:Prod} we obtain $\bigl\| [(ab)c] - abc \bigr\| = 0$, and similarly $\bigl\| [a(bc)] - abc \bigr\| = 0$, concluding that $[(ab)c] = [a(bc)]$.

  Proceeding in this manner, we show that $\bigl\| [P(\bar a)] - P(\bar a) \bigr\| = 0$ for every $\bar a \in K$, polynomial $P(\bar X) \in \bZ[\bar X]$ and ring language term $[P]$ which evaluates to $P$ in rings.
  In particular $[P(\bar a)]$ only depends on $P$ and not on the choice of $[P]$, whence it follows that $K$ is a ring.
  If $a \in K \setminus \{0\}$ then $\|a\| = \|a - 0\| > 0$, so $b = \frac{1}{a}$ exists.
  Thus $\bigl\| [ab] - 1 \bigl\| = 0 = \bigl\| [1] - 1 \bigl\|$, whereby $\bigl\| [ab] - [1] \bigl\| = 0$ and $[ab] = [1]$, so $K$ is a field.

  The identity $\bigl\| [P(\bar a)] - P(\bar a) \bigr\| = 0$ also implies that $\bigl\| P^*(\bar a) \bigr\| \bigl\| [P(\bar a)] \bigr\| = \bigl\| [P(\bar a)]^* \bigr\| \bigl\| P(\bar a) \bigr\|$, or $\bigl| [P(\bar a)] \bigr| = \frac{\| P(\bar a) \|}{\| P^*(\bar a) \|}$.
  By axiom \autoref{ax:Prod} it follows that $|[ab]| = \frac{\|a\|\|b\|}{\|a^*\|\|b^*\|} = |a||b|$.
  Similarly, with axiom \autoref{ax:Ult} we have $|[a+b]| = \frac{\|a+b\|}{\|a^*|\|b^*\|} \leq \frac{\|b^*\|\|a\| + \|a^*\|\|b\|}{\|a^*|\|b^*\|} = |a| + |b|$.
  It follows that $K$ is a valued field, and that the interpretation of the symbols $\|P\|$ is as intended, completing the proof.
\end{proof}

The problem with extending multiplication to the projective line arises with expressions close to $0 \cdot \infty$, i.e., when trying to multiply points which are close to $0$ with points which are close to $\infty$.
This situation cannot happen when taking powers, and indeed,

\begin{lem}
  \label{lem:DefinablePower}
  For $n \in \bZ$, the operation $x \mapsto x^n$ is uniformly
  definable in models of $MVF$.
  This is under the convention that
  $0^0 = \infty^0 = 1$, $\infty^n = \infty$ for $n > 0$,
  and $\infty^n = 0$, $0^n = \infty$ for $n < 0$.
\end{lem}
\begin{proof}
  Indeed, $\ba^n = [a^n:(a^*)^n]$ and $|a^n| \vee |(a^*)^n| = 1$.
  It follows that
  $d(y,x^n) = \| x^n-y \|$,
  and similarly
  $d(y,x^0) = \|y^*\|$,
  $d(y,x^{-n}) = \| 1 - x^ny \|$.
\end{proof}

It is natural to ask whether other projective spaces $K\bP^n$,
for $n > 1$, have more (or less) structure than the projective line.
In order to give a precise meaning to this question, we should first
define the projective spaces as metric structures.
It will be most convenient to define the entire family
$(K\bP^n)_n$ as a single multi-sorted structure $K\bP$.

\begin{dfn}
  The signature $\cL_\bP$ consists of $\aleph_0$ many sorts
  $\{\bP^n\}_{n\in\bN}$.
  They are equipped with the following symbols:
  \begin{itemize}
  \item For each $n,m$ a function symbol
    $\otimes\colon \bP^n \times \bP^m \to \bP^{n+m+nm}$.
  \item For each $A \in SL_{n+1}(\bZ)$ (or in some generating subset),
    a function symbol
    $A\colon \bP^n \to \bP^n$.
  \item For each $n$ a predicate symbol $\|{\cdot}\|$ on $\bP^n$.
  \end{itemize}
\end{dfn}

\begin{dfn}
  \label{dfn:KP}
  Let $(K,|{\cdot}|)$ be any valued field.
  We define an $\cL_\bP$-pre-structure $K\bP$ as follows:
  \begin{itemize}
  \item The sort $\bP^n$ consists of the projective space
    $K\bP^n$, namely the quotient of $K^{n+1} \setminus\{0\}$ by
    $K^\times$.
    The equivalence class of $(a_0,\ldots,a_n)$ will be denoted
    $\ba = [a] = [a_i]_i = [a_0,\ldots,a_n]$.
    We may, and shall, assume that each representative satisfies
    $\bigvee |a_i| = 1$.
  \item
    For $n,m \in \bN$,
    we fix some natural isomorphism
    $K^{n+1} \otimes K^{m+1} \cong K^{(n+1)(m+1)}$,
    say the one given by $(a \otimes b)_{i + (n+1)j} = a_ib_j$.
    We then interpret $\otimes$ as the Segre embedding
    $[a] \otimes [b] = [a \otimes b] = [a_ib_j]_{i\leq n,j\leq m}$.
  \item
    For $A \in SL_{n+1}(\bZ)$, the corresponding function symbol
    acts on $K\bP^n$ naturally via its action on
    $K^{n+1} \setminus \{0\}$.
  \item We interpret:
    \begin{gather*}
      \bigl\| \ba \bigr\| = |a_0|.
    \end{gather*}
  \item The distance on $K\bP^n$ is interpreted as:
    \begin{gather*}
      d( \ba, \bb )
      = \bigvee_{i<j<n} |a_ib_j-a_jb_i|.
    \end{gather*}
  \end{itemize}
  Notice that on $K\bP^1$, the interpretation of $\|x\|$ and $d(x,y)$
  is consistent with that given in \autoref{dfn:KP1}.
\end{dfn}

We need check that the distance defined above is indeed an ultra-metric distance function.
Clearly it only depends on the equivalence classes $\ba$ and $\bb$.
One checks easily that $d(\ba,\bb) = 0$ if and only if $\ba = \bb$.
Symmetry is immediate.
We are left with checking the ultra-metric triangle inequality.
Let $\ba,\bb,\bc \in K\bP^n$, and fix $j_0$ such that $|b_{j_0}| = 1$.
For all $i$ and $k$ we then have:
\begin{align*}
  |a_ic_k-a_kc_i|
  & = |a_ib_{j_0}c_k - a_{j_0}b_ic_k + a_{j_0}b_ic_k - a_{j_0}b_kc_i + a_{j_0}b_kc_i - a_kb_{j_0}c_i|
  \\ &
  \leq
  |c_k||a_ib_{j_0} - a_{j_0}b_i| \vee |a_{j_0}||b_ic_k - b_kc_i| \vee |c_i||a_{j_0}b_k - a_kb_{j_0}|
  \\ &
  \leq d(\ba,\bb) \vee  d(\bb,\bc).
\end{align*}

In order to show that $\bP^n$ is interpretable in $\bP^1$ we shall attempt to repeat the standard trick of covering $\bP^n$ with $n+1$ affine charts.
The problem is that $\bA^n$ is not definable, or even type-definable, in $\bP$, so we shall have to make do with $n+1$ copies of $(\bP^1)^n$ instead.
As above, a point $\ba \in \bP^1$ is viewed as $[a:a^*]$ where $|a|\vee|a^*| = 1$.
It is either equal to $\infty = [1:0]$ or else can be identified with $\frac{a}{a^*} \in K$.
Agreeing that $|\infty| = \infty$ we have $|\ba| \leq 1$ if and only if $|a^*| = 1$.
As in \autoref{lem:DefinablePower} we also have $\ba^{-1} = [a^*:a]$.

Let $M = \frac{n(n+1)}{2}$.
Given a tuple $\bar \ba = (\ba_{ij})_{i < j \leq n} \in (\bP^1)^M$ let $\ba_{ii} = 1 = [1:1]$ and $\ba_{ji} = \ba_{ij}^{-1}$, and consider the matrix
\begin{gather}
  \label{eq:PnMatrix}
  \bigl( \ba_{ij} \bigr)_{i,j \leq n}
  =
  \begin{pmatrix}
    1 & \ba_{0,1} & & \cdots &  & \ba_{0,n} \\
    \ba_{0,1}^{-1} & 1 & & \cdots & & \ba_{1,n} \\
     &  & 1 & & &  \\
    \vdots & \vdots &  & 1 & & \vdots \\
    & &  &  & 1 &  \\
    \ba_{0,n}^{-1} & \ba_{1,n}^{-1} &  & \cdots &  & 1
  \end{pmatrix}
\end{gather}
Intuitively, we wish to consider such matrices whose rows represent identical points in the standard affine charts for $\bP^n$, i.e., such that
\begin{gather*}
  [1:\ba_{0,1}:\ldots:\ba_{0,n}]
  =
  [\ba_{1,0}:1:\ba_{1,2}:\ldots:\ba_{1,n}]
  = \ldots =
  [\ba_{n,0}:\ldots:\ba_{n,n-1}:1].
\end{gather*}
These precise identities are meaningless, since some of the $\ba_{ij}$ may be $\infty$, but we may nonetheless express them formally by the system of equations
\begin{gather*}
  X_{ij}X_{jk} = X_{ik}
  \qquad
  (i < j < k \leq n),
\end{gather*}
which are homogenised into
\begin{gather*}
  X_{ij}X_{jk}X_{ik}^* = X_{ik}X_{ij}^*X_{jk}^*
  \qquad
  (i < j < k \leq n).
\end{gather*}

The following asserts that the solutions to these equations form a well-behaved (definable) set, and that this set covers $\bP^n$.
We recall from \cite{BenYaacov-Berenstein-Henson-Usvyatsov:NewtonMS} or \cite[Fact~1.7]{BenYaacov:DefinabilityOfGroups} that in continuous logic, a subset $X \subseteq M^n$ is called a \emph{definable set} if it is closed and the distance predicate $d(X,\bar x)$ is definable.
This has several equivalent characterisations, among which the existence of a definable predicate $\varphi(\bar x)$ such that $d(X,\bar x) \leq \varphi(\bar x)$ and such that the zero set of $\varphi$ is exactly $X$.
That the latter property implies the former uses quantification, and when dealing with quantifier-free definability the two properties need no longer be equivalent.
The latter one is more robust, and in particular can be shown to still hold if we replaced the ambient distance with an equivalent definable one, so it is it we shall use.

\begin{dfn}
  \label{dfn:QFDefinableSet}
  We shall say that a set $X$ is \emph{quantifier-free definable} if there exists a quantifier-free definable predicate (i.e., a uniform limit of quantifier-free formulae) $\varphi(\bar x)$ such that, first, $X$ is the zero set of $\varphi$, and second, $d(X,\bar x) \leq \varphi(\bar x)$.
\end{dfn}

\begin{lem}
  \label{lem:PnCovering}
  Let $E \subseteq (\bP^1)^M$ consist of all tuples satisfying the homogeneous equations above.
  \begin{enumerate}
  \item The set $E$ is quantifier-free definable.
  \item For every tuple $\bar \ba \in E$ there exists $\ell \leq n$ such that in the $\ell$th row of the matrix \autoref{eq:PnMatrix} all entries are finite of value $\leq 1$.
  \item Let $\ell$ be as in the previous item, and let $\bb \in \bP^n$ be the class of the $\ell$th row, i.e., $\bb = [\ba_{\ell,0}:\ldots:\ba_{\ell,\ell-1}:1:\ba_{\ell,\ell+1}:\ldots:\ba_{\ell,n}]$.
    Then $\bb$ is the unique solution for the following system of homogeneous equations
    \begin{gather*}
      a_{ij}Y_i = a_{ij}^*Y_j \qquad (i<j).
    \end{gather*}
    Conversely, every $\bb \in \bP^n$ arises in this manner (for some $\bar \ba \in E$).
  \end{enumerate}
\end{lem}
\begin{proof}
  We define
  \begin{gather*}
    \varphi(\bar x)
    =
    \bigvee_{i < j < k}
    \|x_{ij}x_{jk} - x_{ik}\|.
  \end{gather*}
  Then $E$ is the zero set of $\varphi$, and we claim that $d(\bar x,E) \leq \varphi(\bar x)$, which is enough for the first item.
  Indeed, assume that $\bar \ba \notin E$, so $\varphi(\bar \ba) = r > 0$, and we wish to show that $d(\bar \ba,E) \leq r$.
  If $r = 1$ then there is nothing to show.
  We may therefore assume that $r < 1$.
  It will be convenient to work with the entire matrix \autoref{eq:PnMatrix} rather than with its upper triangle.
  Observe that passing to the whole matrix does not change our basic hypothesis, i.e., $\bigvee_{i < j < k \leq n} \|\ba_{ij}\ba_{jk} - \ba_{ik}\| = \bigvee_{i,j,k \leq n} \|\ba_{ij}\ba_{jk} - \ba_{ik}\|$.
  If we apply a permutation of $n+1$ both to the rows and columns of the matrix, the resulting matrix will still have the same properties (namely $\ba_{ij} = \ba_{ji}^{-1}$ and $\bigvee_{i,j,k \leq n} \|\ba_{ij}\ba_{jk} - \ba_{ik}\| \leq r$).

  We first claim that if $\varphi(\bar \ba) = r < 1$ then the matrix possesses a row, say the $\ell$th, such that $|\ba_{\ell j}| \leq 1$ for all $j \leq n$.
  In order to prove the claim it will be enough to show that if the $i$th row does not have this property, say because $|\ba_{ij}| > 1$, then in the $j$th row there are strictly more entries than in the $i$th with value $\leq 1$.
  Indeed, assume that $|\ba_{ik}| \leq 1$ and let us show that $|\ba_{jk}| \leq 1$ as well.
  By assumption we have
  \begin{gather*}
    |a_{ij}a_{jk} a_{ik}^* - a_{ik} a_{ij}^* a_{jk}^*|
    = \|\ba_{ij}\ba_{jk} - \ba_{ik}\|
    \leq r < 1.
  \end{gather*}
  We also assume that $|a_{ik}^*| = 1$ and $|a_{ij}^*| < 1 = |a_{ij}|$, whereby $|a_{ij}a_{jk} a_{ik}^*| = |a_{jk}|$ and $|a_{ik} a_{ij}^* a_{jk}^*| \leq |a_{ij}^*| < 1$.
  Since the difference has value $< 1$ we must have $|a_{jk}| < 1$ as well, so $|a_{jk}^*| = 1$ and $|\ba_{jk}| = |a_{jk}| < 1$.
  In addition we have $|\ba_{jj}| = 1 < |\ba_{ij}|$, which is one more, so our claim is proved.

  We next claim that applying a permutation of rows and columns as described earlier, the entire upper triangle can be assumed to consist of elements of value $\leq 1$.
  Indeed, by the previous claim we may assume that $|\ba_{0i}| \leq 1$ for all $i$ and then proceed by induction on $n$ to treat the matrix $(\ba_{ij})_{1 \leq i,j \leq n}$.

  We are now at a situation where $|\ba_{ij}| \leq 1$ if $i < j$ (and $|\ba_{ij}| \geq 1$ if $i > j$).
  We observe that if $\ba,\bb,\bc \in \bP^1$ all have values $\leq 1$ then the product $\ba\bb = [ab:a^*b^*]$ is well defined and moreover $|ab| \leq 1 = |a^*b^*|$, i.e., the $|ab| \vee |a^*b^*| = 1$.
  It follows that
  \begin{gather*}
    d(\ba\bb,\ba\bc)
    = |aba^*c^* - a^*b^*ac|
    \leq |bc^* - b^*c|
    = d(\bb,\bc).
  \end{gather*}
  Similar observations hold if all values are $\geq 1$.
  We may therefore define
  \begin{gather*}
    \bc_{ij} = \prod_{i \leq k < j} \ba_{k,k+1},
    \qquad
    \bc_{ji} = \bc_{ij}^{-1} = \prod_{i \leq k < j} \ba_{k+1,k},
    \qquad
    (i \leq j).
  \end{gather*}
  It is not difficult to check that $\bar \bc \in E$, and in order to prove the first item all that is left to check is that $d(\bar \bc,\bar \ba) \leq r$.
  Keeping in mind that $d(\bc,\ba) = d(\bc^{-1},\ba^{-1})$, it will be enough to check that $d(\ba_{ij},\bc_{ij}) \leq r$ for all $i<j$.
  We do this by induction on $j-i$.
  In the base case $j-i = 1$ we have $\ba_{ij} = \bc_{ij}$.
  Assume now that $d(\ba_{ij},\bc_{ij}) \leq r$.
  Then
  \begin{align*}
    d(\bc_{i,j+1},\ba_{i,j+1})
    &
    \leq
    d(\bc_{i,j+1},\ba_{ij}\ba_{j,j+1})
    \vee
    d(\ba_{ij}\ba_{j,j+1},\ba_{i,j+1})
    \\ &
    =
    d(\bc_{ij}\ba_{j,j+1},\ba_{ij}\ba_{j,j+1})
    \vee
    \|\ba_{ij}\ba_{j,j+1} - \ba_{i,j+1}\|
    \\ &
    \leq
    d(\bc_{ij},\ba_{ij})
    \vee
    r
    = r.
  \end{align*}
  This concludes the proof of the first item, and we have also proved the second item as a special case of our first claim.

  For the third item, the fact that $[\ba_{\ell,0}:\ldots:\ba_{\ell,\ell-1}:1:\ba_{\ell,\ell+1}:\ldots:\ba_{\ell,n}]
  \in \bP^n$ is a solution is an immediate consequence of the hypothesis that $\bar \ba \in E$.
  Conversely, let $\bb \in \bP^n$ be any solution.
  Then $b_i = \ba_{\ell i}b_\ell$ for all $i$, and since $\bigvee |b_i| = 1$ we must have $|b_\ell| = 1$.
  We may therefore assume that $b_\ell = 1$ and we obtain $b_i = \ba_{\ell i}$ as desired.
  Finally, let $\bb \in \bP^n$, and define $\ba_{ij} = [b_j:b_i]$ when at least one of $b_i,b_j$ is non zero and $[1:1]$ otherwise.
  Then $\bar \ba \in E$ and $\bb$ is the associated solution.
\end{proof}

We recall from \cite[Section~1.2]{BenYaacov:DefinabilityOfGroups} that a map $f\colon X \rightarrow Y$ between type-definable subsets of a structure is called \emph{definable} if its graph is type-definable, or equivalently, if composing any definable predicate with $f$ yields a definable predicate (a type-definable set is one which is the intersection of a family of zero sets of formulae, or of definable predicates; as in classical logic, a type-definable set corresponds to a closed set of types, see \cite[Section~1.1]{BenYaacov:DefinabilityOfGroups}).
The former characterisation implies that if $f$ is bijective then its inverse is definable as well.
In the latter characterisation, it suffices to verify for the distance predicate alone.

\begin{thm}
  \label{thm:ProjSpaceInterpretation}
  The projective line $K\bP^1$ is uniformly quantifier-free biïnterpretable with $K\bP$, and in fact $K\bP^1$ is uniformly definable (rather than merely interpretable) in each of the sorts $K\bP^n$ of $K\bP$ for $n \geq 1$.
  More precisely:
  \begin{enumerate}
  \item
    The $\cL_{\bP^1}$-structure $K\bP^1$ and the sort $\bP^1$ of the $\cL_\bP$-structure $K\bP$ are quantifier-free definable in one another, meaning that a predicate $\varphi\colon (K\bP^1)^m \to [0,1]$ is quantifier-free definable in $K\bP^1$ if and only if it is quantifier-free definable in $K\bP$.
  \item
    For every $n \geq 1$ there exist a quantifier-free definable subset $D_n \subseteq \bP^n$ and a definable bijection $\theta_n\colon D_n \to \bP^1$ such that for every quantifier-free definable predicate $\varphi \colon (\bP^1)^m \to [0,1]$, the predicate $\varphi \circ (\theta_n)\colon (D_n)^m \to [0,1]$ is quantifier-free definable as well.
  \item
    For every $n$ there exist a quantifier-free definable subset $E_n \subseteq (\bP^1)^{M(n)}$ and a definable surjection $\rho_n\colon E_n \to \bP^n$ such that for every quantifier-free definable predicate $\varphi \colon
    \bP^{n_0} \times \cdots \times \bP^{n_{m-1}} \to [0,1]$, the predicate $\varphi \circ (\rho_{n_0},\ldots,\rho_{n_{m-1}})\colon
    E_{n_0} \times \cdots \times E_{n_{m-1}} \to [0,1]$ is quantifier-free definable as well.
  \item The predicates defining $D_n$ and $E_n$, as well as the translation schemes from quantifier-free predicates in one sort or structure to another are uniform, i.e., do not depend on $K$.
  \end{enumerate}
\end{thm}
\begin{proof}
  The first item is easy, keeping in mind that it is enough to show
  that every atomic formula in one structure is quantifier-free
  definable in the other.

  For the second item, we let
  $D_n = \{ [a_0:a_1:0:\ldots:0] \colon [a_0:a_1] \in \bP^1\}$.
  It is not difficult to check that
  $d(\bb,D_n) = \bigvee_{2\leq i \leq n} |b_i|$
  which is definable by a quantifier-free formula.
  The map $\theta_n \colon [a_0:a_1:0:\ldots:0] \mapsto [a_0:a_1]$ is definable since its graph is given by
  \begin{gather*}
    \theta_n(x) = y
    \quad \Longleftrightarrow \quad
    \|x_0y_1 - x_1y_0\| = 0.
  \end{gather*}
  We leave it to the reader to check that the pull-back of every
  atomic formula in $\bP^1$ is quantifier-free definable
  in $\bP^n$.

  For the third item most of the work has already been done
  in \autoref{lem:PnCovering}.
  We take $M(n) = \frac{n(n+1)}{2}$
  and define $E_n$ as in the Lemma.
  Then we have already seen that $E_n$ is quantifier-free definable
  and constructed the surjection $\rho_n\colon E_n \to\bP^n$.
  Again we leave it to the reader to check that the pull-back of
  an atomic formula from $\prod \bP^{n_i}$
  to $\prod E_{n_i}$ is quantifier-free definable.

  Everything we did (or left to the reader) is independent of the
  field $K$, whence follows the uniformity.
\end{proof}

It follows that the class of structures $K\bP$ is elementary as well.
Moreover, if we prove that some theory extending $MVF$ eliminates quantifiers (as we shall, in \autoref{thm:QE} below) it will follow that the corresponding $\cL_\bP$-theory eliminates quantifiers as well.

\section{The theory of algebraically closed metric valued fields}

\begin{dfn}
  We define $ACMVF$, the theory of algebraically closed metric valued
  fields, to consist of $MVF$ along with the following
  additional axioms
  \begin{align*}
    & \exists y \, \|y\| = \half
    \\
    & \exists y \,
    \|P(\bar x,y)\| = 0
    && (\deg_Y(P) \geq 1)
  \end{align*}
\end{dfn}
As usual, the existential quantifier should be understood in the
approximate sense.
In the case of the first axiom, it may indeed happen that in a model
of $ACMVF$ the value $\half$ never occurs.
For the second axiom, the approximate witnesses must accumulate near
at least one of finitely many roots, so a root must exist in the
(complete) model.

\begin{lem}
  The models of $ACMVF$ are precisely the projective lines over
  complete, algebraically closed, non trivially valued fields.
\end{lem}
\begin{proof}
  One direction is clear.
  For the other, given an algebraically closed field equipped with a
  non trivial valuation, the set of values must be dense in $\bR$ and
  in particular contain $\half$ in its closure.
\end{proof}

\begin{fct}
  Let $K \subseteq L$ be an extension of valued fields,
  where $K$ is complete, and let $a \in L$ be algebraic over $K$ of
  degree $n$ and with irreducible polynomial $P(X) \in K[X]$.
  Then $|a|^n = |P(0)|$.
\end{fct}

\begin{thm}
  \label{thm:QE}
  The theory $ACMVF$ eliminates quantifiers.
  It is therefore the model completion of $MVF$.
\end{thm}
\begin{proof}
  Let both $K\bP^1,F\bP^1 \vDash ACMVF$ be somewhat saturated, and let $\theta\colon A\to B$ be a valuation-preserving isomorphism of relatively small sub-fields $A \subseteq K$ and $B \subseteq F$.
  First of all we may assume that $A$ and $B$ are complete.
  Second, any extension of the isomorphism to an algebraic isomorphism of their algebraic closure will preserve the valuation, so we may further assume that $A$ and $B$ are algebraically closed (of course, the algebraic closure need not be complete, so we would have to pass to the completion again).

  Let now $c \in K$ be transcendental over $A$.
  The quantifier-free type of $c$ over $A$ is determined by the mapping assigning to each $P(X) \in A[X]$ the value $|P(c)|$.
  Since $A$ is algebraically closed, it suffices to know this for linear polynomials, i.e., to know $|c-a|$ for all $a \in A$.

  For our purposes it will be enough to show that for every finite tuple $a_0,\ldots,a_{n-1} \in A^n$ and every $\varepsilon > 0$ there exists $d \in F$ such that $\bigl| |c-a_i| - |d-\theta a_i| \bigr| < \varepsilon$ for $i < n$.
  Let $r = \min_{i<n} |c-a_i|$.
  Possibly decreasing $\varepsilon$ and re-arranging the tuple $\bar a$, we may assume that there is $k$ such that $|c - a_i| = r$ if $i < k$ and $|c - a_i| > r+\varepsilon$ if $k \leq i < n$.
  It will therefore be enough to find $d \in F$ such that $\bigl| r - |d-\theta a_i| \bigr| < \varepsilon$ for $i < k$ (since then $|d- \theta a_i| = |a_0-a_i| = |c-a_i|$ follows for $k \leq i < n$).
  We consider two cases:

  \textbf{Case I:}
  If $|c| > r$, we choose $d_0 \in F$ such that $r < |d_0| < \min(r+\varepsilon,|c|)$ (such $d_0$ exists since the set of values is dense in $\bR$), and let $d = d_0+\theta a_0$.
  Then $|d - \theta a_i| = |d_0|$ for all $i<k$.

  \textbf{Case II:}
  If $|c| \leq r$, then $|a_i| \leq r$ for all $i < k$.
  Since $B$ is algebraically closed, so is its residue field.
  In particular, the residue field is infinite, so we may choose $b_{\leq k} \in B$ such that $|b_i| = 1$ for all $i\leq k$ and $|b_i-b_j| = 1$ for all $i<j\leq k$.
  We may also choose $e \in F$ such that $r-\varepsilon < |e| < r$.
  We claim that there is $j \leq k$ such that for all $i < k$: $|b_je - \theta a_i| \geq |e|$.
  Indeed, otherwise, by the pigeonhole principle we can find $i < j \leq k$ such that $|b_ie - b_je| < |e|$, whereby $|b_i-b_j| < 1$, contrary to our assumption.
  Let $d$ be this $b_je$.
  Since $|a_i| \leq r$ and $|d| < r$, we must have $|e| \leq |d-\theta a_i| \leq r$ for all $i < k$.

  This concludes the proof that $K$ and $F$ correspond by an infinite back and forth.
  It follows that $ACMVF$ eliminates quantifiers.
  It is also clearly a companion of $MVF$ and therefore it is its model completion.
\end{proof}

\begin{rmk}
  Let $MVF_\bZ$ denote the theory $MVF$ along with axioms saying that the set of non zero values is contained in some fixed infinite discrete group, say $e^\bZ$.
  This can be expressed by the axiom $\|x^*\| \in e^{-\bN} \cup \{0\}$.
  In models of this theory both the valuation ring and its complement are type-definable, so they are in fact definable.
  The maximal ideal is definable as well, so we may refer to the residue field directly as an imaginary sort.
  Similarly, for every $n$, the set of field elements of value $e^{-n}$ is definable.

  Let $ACMVF_\bZ$ consist in addition of axioms saying that the value $e^{-1}$ is attained, that every element of value $e^{kn}$ has an $n$th root and that every irreducible monic polynomial over the valuation ring with free term $1$ has a root.
  Then $ACMVF_\bZ$ eliminates quantifiers, and it is the model completion of $MVF_\bZ$.
  The argument is similar to that given for \autoref{thm:QE}.
\end{rmk}

\begin{cor}
  The following is an exhaustive list of the completions of $ACMVF$:
  \begin{enumerate}
  \item
    Characteristic $(0,0)$: $|p| = 1$ for all prime $p$.
  \item
    Characteristic $(0,p)$: $|p| = \alpha$ for some prime $p$ and $0 < \alpha < 1$.
  \item
    Characteristic $(p,p)$: $p = 0$ for some prime $p$.
  \end{enumerate}
\end{cor}
\begin{proof}
  It is known (e.g., from \cite{Artin:AlgebraicNumbersAndFunctions}) that every model of $ACMVF$ falls into one of these categories and that none of them is empty.
  Since each of the listed theories determines $|n|$ for each $n \in \bZ$, by quantifier elimination they are complete.
\end{proof}

The space of completions consists therefore of a family of segments $[0,1]$, one for each prime $p$, with all the $1$ points identified (the $(0,0)$ case).
This is essentially the zero dimensional Berkovich space over $\bZ$, just without the segment corresponding to Archimedean valuations, which we chose to exclude.
Similarly,

\begin{cor}
  Let $K$ be a model of model of $ACMVF$, let $A \subseteq K$, and let $K_0$ be the complete sub-field generated by $A$.
  Then the space of $1$-types over $A$ in the sort $\bP^n$ is precisely the $n$-dimensional projective analytic Berkovich space over $K_0$.
\end{cor}

Let us give a slightly different characterisation of types (or more precisely, of $1$-types) which will be useful for counting them.
\begin{dfn}
  Let $K$ be a valued field and let $\cC$ and $\cC'$ be two chains of
  closed balls in $K\bP^1$.
  Say that $\cC$ and $\cC'$ are \emph{mutually co-final} if each ball
  in one chain contains some ball belonging to the other.
  This is an equivalence relation, and by a \emph{sphere} over
  $K\bP^1$ we mean an equivalence class of such a chain.
  The set of all spheres will be denoted $\Sph(K\bP^1)$.

  Let $S,S' \in \Sph(K\bP^1)$ be spheres, say represented by
  $\cC$ and $\cC'$.
  We define the radius of $S$ as
  $\rad(S) = \inf_{B \in \cC} \rad(B)$.
  We define the Hausdorff distance between $S$ and $S'$
  as the limit of Hausdorff distances between balls in
  $\cC$ and $\cC'$:
  \begin{gather*}
    d_H(S,S')
    =
    \mathop{\lim_{B \in \cC, \rad(B) \to \rad(S)}}
    _{B' \in \cC', \rad(B') \to \rad(S')}
    d_H(B,B').
  \end{gather*}
\end{dfn}
It is not difficult to see that for closed balls $B$ and $B'$,
\begin{itemize}
\item $d_H(B,B') = 0$ if and only if $B = B'$,
\item if $B \supsetneq B'$ then $d_H(B,B') = \rad(B)$, and
\item if $B \cap B' = \emptyset$ then
  $d_H(B,B') = d(B,B')$.
\end{itemize}

Notice that every sphere admits a countable representative.
The field $K$ is complete if and only if every sphere of radius
zero contains a point.
If every sphere contains a point then $K$ is called
\emph{spherically complete}.

\begin{thm}
  Let $K\bP \vDash ACMVF$.
  Then:
  \begin{enumerate}
  \item Let $S \in \Sph(K\bP^1)$ be a sphere, say the class of $\cC = \{\overline B(a_n,r_n)\}_{n \in \bN}$, and let $r = \inf r_n$ denote its radius.
    Then the set of conditions
    \begin{gather}
      \label{eq:pS}
      \{ \|x-a_n\| \leq r_n\}_{n\in\bN}
      \cup \{ \|x-a\| \geq r\}_{a \in K\bP^1}
    \end{gather}
    axiomatises a complete type $p_S(x) \in \tS_1(K)$ which depends only on $S$.
  \item The mapping $S \mapsto p_S$ is an isometric bijection $\bigl( \Sph(K\bP^1), d_H \bigr) \cong \bigl( \tS_1(K),d \bigr)$, where the distance between two types is the minimal distance between realisations.
  \end{enumerate}
\end{thm}
\begin{proof}
  Let us first show that \autoref{eq:pS} is consistent for every $S$.
  Possibly passing to a sub-sequence, and possibly applying the isometry $a \mapsto a^{-1}$ to $\bP^1$, we may assume that $|a_n| \leq 1$ for all $n$.
  Let $L = K(\alpha)$ where $\alpha$ is transcendental over $K$.
  Then we may extend the valuation to $L$ so that for every polynomial $P(X) = \sum_{k \leq m} b_kX^k \in K[X]$ we have $|P(\alpha)| = \bigvee_k \rad(S)^k |b_k|$.
  In particular, $|\alpha| = \rad(S) \leq 1$.
  Further extending to a model of $ACMVF$ we may assume that $L\bP^1 \succeq K\bP^1$.
  Let $c_n = a_n+\alpha$.
  For $a \in K\bP^1$ we have $\|c_n - a\| = 1 \geq \rad(S)$ if $|a| > 1$ and $\|c_n - a\| = |c_n - a| = |a_n-a| \vee \rad(S)$ otherwise.
  For $m < n$ we also have $\|c_n-a_m\| = |\alpha + (a_n-a_m)\| \leq r_m$.
  Thus \autoref{eq:pS} is finitely consistent and therefore consistent.

  By quantifier elimination and the fact that $K$ is algebraically closed, the type of an element $\alpha$ over $K\bP^1$ is determined by $|\alpha - a|$ as $a$ varies over $K$, or equivalently, by $\|\alpha - a\|$ as $a$ varies over $K\bP^1$.
  Let $S$ be the sphere consisting of all balls $\overline B\bigl( a, d(a,\alpha) \bigr)$, $a \in K\bP^1$.
  Then $S$ only depends on $\tp(\alpha/K)$, and conversely, $p_S = \tp(\alpha/K)$.
  This yields the bijection $\Sph(K\bP^1) \to \tS_1(K)$.

  It is left to show that this bijection is isometric.
  So let $S$ and $S'$ be two distinct spheres and let $\alpha$ and $\beta$ realise $p_S$ and $p_{S'}$, respectively.
  Assume first that $B \cap B' \neq \emptyset$ for all $B \in S$ and $B' \in S'$.
  Then $\rad(S) \neq \rad(S')$ (since else the spheres coincide), say $\rad(S) > \rad(S)$.
  Then $d_H(S,S') = \rad(S) = d(\alpha,\beta)$.
  On the other hand, if there are $B \in S$ and $B' \in S'$ which are disjoint then $d_H(S,S') = d(B,B') = d(\alpha,\beta)$ again.
\end{proof}

\begin{cor}
  \label{cor:StrictlyStable}
  The theory $ACMVF$ is strictly stable (i.e., stable non super-stable).
\end{cor}
\begin{proof}
  Let $K$ be a model.
  Since every sphere has a countable representative, a quick calculation yields that there are at most $|K|^{\aleph_0}$ spheres, and therefore types, over $K$.
  Thus the theory is stable.

  On the other hand, for every $0 < r < r' < 1$, every ball of radius $r'$ contains $|K|$ many distinct balls of radius $r$.
  Thus a refinement of our earlier calculation yields that there exist precisely $|K|^{\aleph_0}$ distinct spheres of radius $r$.
  The distance between any two such spheres is at least $r$, so the theory is not super-stable.
\end{proof}

\begin{rmk}
  \label{rmk:StrictlyStablePerturbation}
  Here we assume the reader has some familiarity with the notion of perturbations of metric structures and its uses, as introduced in \cite{BenYaacov:Perturbations}, or, in a somewhat simpler fashion, in \cite{BenYaacov:TopometricSpacesAndPerturbations}.
  Extensions of perturbations to types over parameters, and $\lambda$-stability up to perturbation, are also discussed in \cite{BenYaacov:TopometricSpacesAndPerturbations}.
  For example, it is shown in \cite{BenYaacov-Berenstein:HilbertProbabilityAutmorphismPerturbation} that the theory of atomless probability algebras with a generic automorphism, even though it is strictly stable, is $\aleph_0$-stable up to arbitrarily small perturbations of the automorphism.

  Omitting many details, let us consider a theory $T$ and a set of parameters $A \subseteq \cM \vDash T$.
  We define $\cL(A)$ to consist of the base language $\cL$ together with, for each $a \in A$, a unary predicate $P_a(x)$ for the distance $d(a,x)$.
  Thus ``a model of $T$ containing $A$'' is essentially the same as a model of $T(A) = \Th_{\cL(A)}(\cM)$, and types over $A$ are just types of $T(A)$ over $\emptyset$.
  Roughly speaking, a perturbation of a model of $T(A)$ consists of modifying the interpretation of the symbols of $\cL$ (usually with some small uniform bound on the extent of the modification, prescribed by a \emph{perturbation system}), in such a manner that the end result is again a model of $T(A)$, and that the predicates $P_a$, representing the parameters, remain unchanged.

  In $ACVMF$, when $A = K$ is a model, a $1$-type $\tp(b/K)$ is entirely determined by the map $a \mapsto P_a(b)$, so a perturbation cannot change $1$-types over $K$ at all (even if it does change, to some small extent, the distance and/or algebraic structure of an extension of $K$ containing the realisation).

  It follows that even up to perturbation, in the sense of the articles cited above, $ACMVF$ is strictly stable, i.e., $\lambda$-stable up to perturbation only when $\lambda = \lambda^{\aleph_0}$.
\end{rmk}

The same argument does not work for $ACMVF_\bZ$, since there a strictly decreasing sequence of radii must necessarily go to zero, and it follows that the theory is $\aleph_0$-stable.
This is hardly surprising, since equal characteristic models of $ACMVF_\bZ$ are just something of the form $K = k((X))$.
They are therefore interpretable in the valuation ring $k[[X]]$ which is in turn interpretable (as a metric structure) in $k$, a plain strongly minimal algebraically closed field.

It is an easy fact that if the union of two disjoint type-definable
sets is definable then each of the two sets is definable as well.
The following is a useful extension of this fact.

\begin{lem}
  \label{lem:DefinableUnionIntersection}
  Let $X$ and $Y$ be two type-definable sets such that both
  $X \cup Y$ and $X \cap Y$ are definable.
  Then $X$ and $Y$ are definable as well.
\end{lem}
\begin{proof}
  It will be enough to show that $X$ is definable,
  and for this, it will be enough to show
  that for every $\varepsilon > 0$, the $\varepsilon$-neighbourhood
  $B(X,\varepsilon)$ contains a logical neighbourhood of $X$.

  Since $Y$ is type-definable and
  $X \cap Y$ definable,
  the properties $d(x,Y) \leq \delta$ and
  $d(x,X\cap Y) \geq \varepsilon$ are type-definable.
  By compactness there exists $\delta > 0$ such that
  $\bigl( x \in X \textbf{ and }
  d(x,X\cap Y) \geq \varepsilon
  \textbf{ and }
  d(x,Y) \leq \delta \bigr)$
  is contradictory.
  We may further assume that $\delta \leq \varepsilon$.
  We claim that the desired neighbourhood of $X$ is the given by the
  property
  \begin{gather*}
    \Bigl(
    d(x,Y) > \delta
    \textbf{ and } d(x,X \cup Y) < \delta
    \Bigr)
    \textbf{ or }
    d(x,X \cap Y) < \varepsilon.
  \end{gather*}
  Indeed, this is an open property,
  and it holds for every $x \in X$ by choice of $\delta$.
  Assume this property holds for $x$.
  If $d(x,X \cap Y) < \varepsilon$
  then $d(x,X) < \varepsilon$ as well.
  Otherwise, $d(x,X \cup Y) < \delta$ and
  $d(x,Y) > \delta$ imply that $d(x,X) < \delta \leq \varepsilon$,
  and the proof is complete.
\end{proof}

The following generalises the fact that a definable image of a
definable set is definable.

\begin{lem}
  \label{lem:DefinablePartialImage}
  Let $X$ be a definable set, $Y$ and $Z \subseteq X$ type-definable
  sets, and let $f\colon X \setminus Z \to Y$ be a bijection.
  Assume furthermore that $f$ is definable, in the sense that
  there exists a type-definable set $R \subseteq X \times Y$
  such that $R \cap \bigl( (X \setminus Z) \times Y \bigr)$
  is the graph of $f$.
  Then $Y$ is definable as well.
\end{lem}
\begin{proof}
  Since $Y$ is type-definable, the property
  $d(y,Y) \leq r$ is type definable.
  It will therefore be enough to show that $d(y,Y) \geq r$
  is a type-definable property for all $r$.
  Let $\pi(x)$ be the partial type defining $Z$,
  and let $\varphi \in \pi$.
  For each $x \in X$, either $f(x)$ is well defined or
  $\varphi(x) = 0$, so either way $d(y,f(x))\wedge\varphi(x)$ is well
  defined, and we claim that it is a definable predicate.
  Indeed, $d(y,f(x)) \wedge \varphi(x) \geq s$
  if and only if there exists $w$ such that
  $R(x,w)$ and $d(y,w) \wedge \varphi(x) \geq s$,
  and similarly for $\leq s$.
  Since $X$ is definable, we obtain a definable predicate
  \begin{gather*}
    \psi_\varphi(y)
    =
    \inf_{x\in X}
    \bigl[ r \dotminus d(y,f(x)) \bigr] \wedge \varphi(x).
  \end{gather*}
  We conclude by observing that
  $d(y,Y) \geq r$
  is defined by the partial type
  $\{\psi_\varphi\}_{\varphi \in \pi}$.
\end{proof}

Recall:
\begin{fct}[Noether's Normalisation Lemma]
  Let $A$ be an integral domain, finitely
  generated over a field $k$.
  Then there exist algebraically independent elements
  $x_0,\ldots,x_{d-1} \in A$ such that
  $A$ is integral over $k[x_0,\ldots,x_{d-1}]$.

  Moreover, if $k$ is infinite and $A = k[y_0,\ldots,y_{n-1}]$
  then each $x_i$ can be taken to be a
  $k$-linear combination of the $y_j$.
\end{fct}
Let $V$ be a projective variety of dimension $d$
defined over an infinite field $k$.
Let $y = [y_0:\ldots:y_n]$ be a generic point of $V$.
Let $x_0,\ldots,x_d$ be a transcendence basis for
$k[\bar y]$ consisting of $k$-linear combinations of $\bar y$, as per
Noether's Normalisation Lemma.
Then $[x_0:\ldots:x_d:y_0:\ldots:y_n]$
is the generic point of a projective variety isomorphic to $V$.

\begin{prp}
  Let $K\bP^1 \vDash ACMVF$.
  Then every Zariski closed set $V \subseteq K\bP^n$
  is definable.
\end{prp}
\begin{proof}
  Since a finite union of definable sets is definable, we may assume
  that $V$ is a variety, say of dimension $d$.
  Clearly every algebraic morphism is definable, and recall that
  the image of a definable set by a definable mapping is definable as
  well.
  It follows that we may replace $V$ with any isomorphic projective
  variety.
  Therefore, using Noether's Normalisation Lemma we may assume that
  the homogeneous prime ideal defining $V$ is
  $I(V) \subseteq K[X_0,\ldots,X_d,Y_0,\ldots,Y_{n-1}]$,
  where $K[\bar X] \cap I(V) = 0$
  and for each $j < n$ 
  there exists a homogeneous polynomial
  $f_j \in I(V) \cap K[\bar X,Y_j]$
  which is monic in $Y_j$.
  Possibly replacing $V$ with an isomorphic variety we may further
  assume that all the coefficients in each $f_j$ have value $\leq 1$.
  Thus we may express $|f_j(x_0,\ldots,x_d,x_{d+j+1})|$
  as an atomic formula $\|f_j(x)\|$
  in the free variable
  $x = [x_0:\ldots:x_n] \in \bP^n$
  and with parameters in $K$.
  We may further assume that all the $f_j$ have common degree $m$.

  As a first approximation, let
  $J = \langle f_j \rangle_{j<n} \subseteq I(V)$
  be the generated homogeneous
  ideal, and let us show that $V(J)$ is definable.
  Clearly $V(J)$ is the zero set of the formula
  $\bigvee_{j<n} \|f_j(x)\|$,
  and it will be enough to show that
  $d(x,V(J)) \leq \bigvee_{j<n} \|f_j(x)\|^{\frac{1}{m}}$.
  So let us fix  $x \in \bP^n$.
  For $j < n$, let
  \begin{gather*}
    g_j(Y_j) = f_j(x_0,\ldots,x_d,Y_j)
    = \prod_{k<m} (Y_j-\gamma_j^k)
    \in K[Y_j]
  \end{gather*}
  We may assume that for each $j < n$,
  the root $\gamma_j^0 = \gamma_j$ is closest to
  $x_{d+j+1}$ among all the roots of $g_j$.
  Let
  \begin{gather*}
    y = [x_0:\ldots:x_d:\gamma_0:\ldots:\gamma_{n-1}]
    = \left[
      \sfrac{x_0}{s}:\ldots:\sfrac{x_d}{s}:
      \sfrac{\gamma_0}{s}:\ldots:\sfrac{\gamma_{n-1}}{s}
    \right]
    \in V(J),
  \end{gather*}
  where $s$ is chosen of maximal value among
  $x_0,\ldots,x_d,\gamma_0,\ldots,\gamma_{n-1}$.
  A quick calculation yields,
  for $i \leq d$ and $j < n$,
  \begin{gather*}
    |x_iy_{d+j+1} - x_{d+j+1}y_i|
    =
    |\sfrac{x_i}{s}| |\gamma_j-x_{d+j+1}|
    \leq |g_j(x_{d+j+1})|^{\frac{1}{m}}
    = \|f_j(x)\|^{\frac{1}{m}},
  \end{gather*}
  and for $i,j < n$,
  \begin{align*}
    |x_{d+i+1}y_{d+j+1} - x_{d+j+1}y_{d+i+1}|
    &
    =
    |\sfrac{1}{s}| |\gamma_jx_{d+i+1} - \gamma_ix_{d+j+1}|
    \\ &
    \leq
    |\sfrac{\gamma_j}{s}| |x_{d+i+1} - \gamma_i|
    \vee
    |\sfrac{\gamma_i}{s}| |\gamma_j - x_{d+i+1}|
    \\ &
    \leq
    (\|f_i(x)\| \vee \|f_j(x)\|)^{\frac{1}{m}}.
  \end{align*}
  Thus
  $d(x,V(J)) \leq d(x,y) \leq \bigvee_{j<n} \|f_j(x)\|^{\frac{1}{m}}$,
  as desired.

  By construction, $V(J)$ is of dimension
  $\leq d$, and can be decomposed as
  $V(J) = V \cup W$ where $W \subseteq \bP^n$ is a Zariski closed as
  well and $\dim(V \cap W) < d$.
  By induction on the dimension we may assume already known that
  $V \cap W$ is definable.
  We may now apply \autoref{lem:DefinableUnionIntersection} and conclude
  that $V$ is definable.
\end{proof}

\begin{cor}
  Every complete variety is interpretable in $ACMVF$.
\end{cor}
\begin{proof}
  By Chow's Lemma, if $W$ is a complete variety then it is the image
  of a projective variety $V$ by a morphism.
  In other words, it is a definable quotient of a definable set, and
  therefore interpretable.
\end{proof}

In particular, this means that a complete variety $W$ is endowed with
the quotient structure it inherits from the definable set $V$.
This does not depend on the choice of $V$.

\begin{qst}
  Characterise all definable sets over $K$.
  Notice that since every compact set is definable,
  there are definable sets which are not projective varieties,
  e.g., any set of the form $\{a_n\}_n \cup \{0\}$
  where $|a_n| \to 0$.
  More generally, every metrisable totally disconnected
  compact space can be embedded in $K\bP^1$, and a characterisation of
  definable sets will have to allow for them.

  Let $\{V_\alpha\}_{\alpha \in A}$ be a family of projective
  varieties, and assume that for every $\varepsilon > 0$
  there is a finite $A_0 \subseteq A$ such that
  $\bigcup_{\alpha \in A} V_\alpha$ is contained in the
  $\varepsilon$-neighbourhood of $\bigcup_{\alpha \in A_0} V_\alpha$.
  Then $X = \overline{\bigcup_{\alpha \in A} V_\alpha}$ is a definable
  set.
  Every Zariski closed set and every compact set are of this form.
  Are there any other definable sets?
\end{qst}

\begin{qst}
  Let $A$ be any semi-normed ring.
  Let $\cL_\bP(A)$ consist of a constant symbol in the sort $\bP^1$ for each member of $A$, and let $ACMVF(A)$ be the $\cL_\bP(A)$-theory consisting of $ACMVF$ along with axioms saying that $1 = 1_A$, $a+b = (a+_A b)$, $a \cdot b = (a \cdot_A b)$ and $|a| \leq |a|_A$ (i.e., $\|a\| \leq |a|_A$ if $|a|_A < 1$ and $\|a^*\| \geq |a|_A^{-1}$ otherwise).

  Assuming that $I \subseteq A[X_0,\ldots,X_n]$ is a homogeneous ideal, is $V(I)$ uniformly definable in $ACMVF(A)$?
\end{qst}

\section{Real closed and ordered metric valued fields}

We shall now seek to understand the metric valued analogue of the
theory of real closed fields.
First of all, we observe that the class of metric valued fields which
are, as pure fields, formally real, is \emph{not} elementary.
Indeed, such fields can be constructed with $1 + a^2$
of arbitrarily small (non zero) valuation, and in an ultra-product we
would obtain $1 + a^2 = 0$.
Thus $|1+x^2|$ must be bounded away from zero, which, in a real closed
field
(and more generally, in a field where a sum of squares is a square),
implies $|1+x^2| \geq 1$.

\begin{dfn}
  We say that a valued field $(K,|{\cdot}|)$
  is a \emph{formally real valued field},
  or that that $|{\cdot}|$ is a \emph{formally real valuation} on $K$,
  if its residue field is formally real.
  If in addition $K$ is real closed (as a pure field) then we say that
  it is a \emph{real closed valued field}.

  We recall that a field ordering (possibly partial) is one in which
  sums and products of positive elements, as well as all squares, are
  positive.
  A \emph{valued field ordering} is one in which, in addition, the
  valuation ring is convex.
\end{dfn}

\begin{lem}
  \label{lem:FormallyRealValuedField}
  Let $(K,|{\cdot}|)$ be a valued field.
  Then the following are equivalent.
  \begin{enumerate}
  \item The valued field  $(K,|{\cdot}|)$  is formally real
    (as a valued field).
  \item For all $x_0,\ldots,x_{n-1} \in K$:
    \begin{gather*}
      \left| \sum x_i^2 \right| = \bigvee |x_i|^2.
    \end{gather*}
  \item For all $x_0,\ldots,x_{n-1} \in K$:
    \begin{gather*}
      \left| 1+ \sum x_i^2 \right| \geq 1.
    \end{gather*}
  \end{enumerate}
  Similarly, a field $K$ equipped with a valuation $|{\cdot}|$ and an
  ordering $\leq$ is an ordered valued field if and only if for every
  $x,y \geq 0$: $|x+y| = |x| \vee |y|$.
\end{lem}
\begin{proof}
  Easy.
\end{proof}

A formally real valued field is formally real as a plain field,
and conversely, a field $K$ is formally real if and only if the
trivial valuation on $K$ is formally real.

\begin{lem}
  \label{lem:RealClosedValuedField}
  Let $(K,|{\cdot}|)$ be a complete valued field.
  Then the following are equivalent.
  \begin{enumerate}
  \item The valued field $(K,|{\cdot}|)$ is real closed
    (as a valued field).
  \item The valued field $(K,|{\cdot}|)$ is formally real
    (as a valued field) and maximal as such among its algebraic valued
    field extensions.
  \end{enumerate}
\end{lem}
\begin{proof}
  One direction is immediate.
  For the other, we already know that $(K,|{\cdot}|)$ is a formally
  real valued field, and it is left to show that it is real closed as
  a pure field.
  Indeed, let $K_1/K$ be any proper algebraic field extension,
  which we may assume to be finite.
  We may then equip $K_1$ with an extension of the valuation
  (which is moreover unique since $K$ is complete).
  Let $k_1/k$ denote the corresponding residue field extension.
  Then $(K_1,|{\cdot}|)$ is not formally real, whereby $k_1$ is not
  formally real.
  On the other hand, $k_1/k$ is an algebraic extension, so $k_1$ is
  algebraically closed.
  Since $(K_1,|{\cdot}|)$ is complete, as a finite extension of a
  complete valued field, by Hensel's Lemma we have $i \in K_1$, and in
  particular $K_1$ is not formally real.
  This completes the proof.
\end{proof}

\begin{lem}
  \begin{enumerate}
  \item A real closed valued field admits a unique ordering
    (as a valued field), namely its unique ordering as a pure real
    closed field: $x \geq 0$ if and only if $x$ is a square.
  \item Every formally real valued field embeds in a real closed
    valued field.
  \item A valued field $(K,|{\cdot}|)$ is formally real if and only if
    it admits an ordering (as a valued field).
  \end{enumerate}
\end{lem}
\begin{proof}
  For the first item, all we need to check is that valuation ring is
  convex in the unique field ordering, which is more or less immediate
  from the definition.
  The second item follows from \autoref{lem:RealClosedValuedField}.
  For the third and last item, one direction follows from the previous
  item, the other directly from the definitions.
\end{proof}

In order to express in $\cL_{\bP^1}$ that the valuation is formally
real one needs to take into account the homogenisation, yielding
\begin{gather*}
  \label{ax:FR}
  \tag{FR}
  \left\| \sum x_i^2 \right\|
  =
  \bigvee \|x_i\|^2\prod_{j \neq i} \|x_j^*\|^2.
\end{gather*}
Working in the projective space $\bP^n$ one can express this slightly
more elegantly as
\begin{gather*}
  \tag{FR'}
  \left\| \sum x_i^2 \right\|
  = 1,
\end{gather*}
where the sum is now over the homogeneous coordinates of a single
point $x$.

\begin{dfn}
  We define $FRMVF$, the theory of formally real metric valued fields, to consist of $MVF$ along with the axiom \autoref{ax:FR}.
  We define $RCMVF$, the theory of real closed metric valued fields, to consist, in addition, of the axioms
  \begin{align*}
    & \exists y \, \|y\| = \half, \\
    & \exists y \, \|x^2 - y^4\|, \\
    & \exists y \, \left\| y^{2n+1} + \sum_{i\leq 2n} x_iy^i \right\|.
  \end{align*}
\end{dfn}
As in the discussion following the definition of $ACMVF$, the
existential quantifiers are approximate, but in the case of the second
and third axiom they imply exact existence.

\begin{prp}
  Models of $FRMVF$ ($RCMVF$) are the projective lines over
  complete formally real (real closed and non trivial) valued fields.
\end{prp}

Ordered metric valued fields will be considered in an expanded
language $\cL_{o\bP^1} \supseteq \cL_{\bP^1}$ which we now define.
First, we wish to introduce a predicate $\< x \>$,
equal to zero if and only if $x$ is positive or zero.
Since $\infty$ is neither strictly positive not strictly negative, and
may be arbitrarily close both to positive and to
negative field elements, we require $\<\infty\> = 0$.
One natural definition (which later turns out to be correct) is
$\<x\> = \|x\| \wedge \|x^*\|$ for negative $x$, so in particular we
have a natural identity $\< x \> = \< x^{-1} \>$.
Since our language contains no function symbols, it will be convenient
to go further and add, for each polynomial $P \in \bZ[\bar X]$,
a predicate
\begin{gather*}
  \< P(\bar x) \>
  =
  \begin{cases}
    0 & P(\bar x) \geq 0, \\
    \|P(\bar x)\| \wedge \|P^*(\bar x)\|
    & \text{otherwise}.
  \end{cases}
\end{gather*}
In particular, if any $x_i$ is equal to $\infty$ and
$\deg_{X_i} P > 0$ then $\< P(\bar x) \> = 0$ by the ``otherwise''
clause.
Using the assumption that $K$ is an ordered valued field
one verifies that all the new predicates are $1$-Lipschitz.
In what follows, it will be convenient to keep in mind that
$\|P\| \wedge \|P^*\| =
\bigl( |P| \wedge 1 \bigr) \|P^*\|$.

\begin{dfn}
  We define $OMVF$,  the theory of ordered metric valued fields,
  to consist of $MVF$ along with
  \begin{align*}
    \tag{Tot}
    & \< P \> \wedge \< -P\> = 0
    \\
    \tag{AS}
    & \< P \> \vee \< -P\> = \|P\| \wedge \|P^*\|
    \\
    \tag{CA}
    & \<P+Q\>\|P^*Q^*\|
    \leq
    \<Q\>\|P^*(P+Q)^*\|
    \vee
    \<P\>\|Q^*(P+Q)^*\|
    \\
    \tag{CM}
    & \<-PQ\> \geq \<P\>\<Q\>
  \end{align*}
\end{dfn}

We leave it to the reader to check that if $K$ is an ordered valued
field then the associated $\cL_{o\bP^1}$-structure is a model of
$OMVF$, and conversely, that every model of $OMVF$ arises uniquely in
this fashion.


For any field $K$, let
$\Sq^K = \{x^2\}_{x\in K\bP^1} \subseteq K\bP^1$
(where $\infty^2 = \infty$).
For $P(\bar X) \in \bZ[\bar X]$ we consider the following definable
predicate
\begin{gather*}
  \< P(\bar x) \>^{\Sq} = \inf_y\, \|P(\bar x) - y^2\|.
\end{gather*}

\begin{lem}
  \label{lem:PolynomSq}
  For every model $K\bP^1 \vDash MVF$ we have
  \begin{gather*}
    \< P(\bar x) \>^{\Sq} =
    \begin{cases}
      0 & P(\bar x) \in \Sq, \\
      \|P(\bar x)\| \wedge \|P^*(\bar x)\|
      & \text{otherwise}.
    \end{cases}
  \end{gather*}
  In particular, if $x_i = \infty$ and $\deg_{X_i} P > 0$
  then $\<P(\bar x)\>^{\Sq} = 0$.
\end{lem}
\begin{proof}
  Clearly, if $\bar x \in K$ and $P(\bar x) \in \Sq$
  then $\<P(\bar x)\>^{\Sq} = 0$.
  Also, we observe that
  $\|P(\bar x)-0^2\| = \|P(\bar x)\|$ and
  $\|P(\bar x)-\infty^2\| = \|P^*(\bar x)\|$.
  Thus
  $\< P(\bar x) \>^{\Sq} \leq \|P(\bar x)\| \wedge \|P^*(\bar x)\|$,
  and in particular $\< P(\bar x) \>^{\Sq} = 0$ if
  $x_i = \infty$ and $\deg_{X_i} P > 0$.
  It is left to consider the case where $\bar x \in K$ and
  $P(\bar x) \notin \Sq$.
  Indeed, assume that
  $\< P(\bar x) \>^{\Sq} < \|P(\bar x)\| \wedge \|P^*(\bar x)\|$.
  Then there is $z \in K^*$ such that
  $\|P(\bar x) - z^2\| < \|P(\bar x)\| \wedge \|P^*(\bar x)\|$,
  or equivalently
  $|P(\bar x) - z^2| \|z^*\|^2 < |P(\bar x)| \wedge 1$.
  If $|z| \leq 1$ then $|P(\bar x) - z^2|  < |P(\bar x)|$, whereby
  $|P(\bar x)| = |z^2|$;
  and if $|z| > 1$ then
  $|P(\bar x) - z^2| < \|z^*\|^{-2} = |z^2|$,
  and again $|P(\bar x)| = |z^2|$.
  Either way we get
  $|\frac{P(\bar x)}{z^2} - 1^2| < 1 = |\frac{P(\bar x)}{z^2}|$,
  and by Hensel's Lemma, $P(\bar x) \in \Sq$, contrary to our
  assumption.
\end{proof}

\begin{lem}
  \label{lem:DefinableSq}
  In any metric valued field the set $\Sq$ is closed and
  $d(x,\Sq) = \< x \>^{\Sq}$.
  In particular, $\Sq$ is uniformly
  definable across all complete valued fields.
\end{lem}
\begin{proof}
  It is easy to see that $\|x-z^2\| = d(x,z^2)$
  (compare with \autoref{lem:DefinablePower}), whereby
  $d(x,\Sq) = \< x \>^{\Sq}$.
  By \autoref{lem:PolynomSq}, if $x \notin \Sq$ then
  $d(x,\Sq) = \|x\| \wedge \|x^*\| > 0$,
  so $\Sq$ is closed.
\end{proof}

\begin{prp}
  Let $K \vDash RCMVF$.
  Then $K$ admits a unique expansion to a model of $OMVF$,
  given by $\< P \> = \< P \>^{\Sq}$.
\end{prp}

\begin{thm}
  The $\cL_{o\bP^1}$-theory $ORCMVF = RCMVF \cup OMVF$ is complete and
  admits quantifier elimination.
  The theory $RCMVF$ is model complete.
\end{thm}
\begin{proof}
  Completeness and model completeness follow quite easily from
  quantifier elimination, so we only prove the latter.
  For this, we shall prove that sufficiently saturated models admit an
  infinite back-and-forth.
  Using the uniqueness of the real closure of an ordered field, and
  proceeding as in the proof of \autoref{thm:QE}, we reduce to the case
  where $K\bP^1$ and $F\bP^1$ are two sufficiently saturated models,
  $A \subseteq K$ and $B \subseteq F$ are relatively algebraically
  closed complete sub-fields, and $\theta\colon A \to B$ is an
  isomorphism.
  In particular, $A$ and $B$ are real closed valued fields.

  Now let $c \in K \setminus A$.
  Its quantifier-free type is determined by the value and sign of
  $P(c)$ as $P(X)$ varies over $A[X]$.
  Since $A$ is real closed, every polynomial decomposes as a product
  of linear factors $X - a$ and irreducible quadratic factors
  $(X-a)^2 + b$, $b > 0$ (and $a,b \in A$).
  In the second case we have $(c-a)^2 + b > 0$ and
  $|(c-a)^2 + b| = |c-a|^2 \vee b$.
  Thus, the quantifier-free type of $c$ is determined by the value and
  sign of $c-a$ as $a$ varies over $A$.
  In order to find $d \in F$ with the corresponding quantifier-free
  type over $B$, it is enough to show that for every $\varepsilon > 0$
  and every finite family $a_0,\ldots,a_{n-1} \in A$
  there is $d \in F$ such that
  $d \leq \theta a_i \Longleftrightarrow c \leq a_i$
  and $\bigl|  |d-\theta a_i| - |c - a_i| \bigr| < \varepsilon$.
  We may assume that $a_i < a_{i+1}$ for $i < n-1$.

  If $c > A$ then the valuation on $A$ is necessary trivial.
  In this case we may take $d \in F$ to be any positive element with
  the same value as $c$ (or at least close enough).
  The case $c < A$ is treated similarly.
  Otherwise, there is $i$ for which $a_i < c_i < a_{i+1}$.
  Translating by $a_i$ and dividing by $a_{i+1}$ we may assume
  that $a_i = 0$ and $a_{n+1} = 1$.
  It will then be enough to find $0 < d < 1$ such that
  $\bigl|  |d| - |c| \bigr|,
  \bigl|  |1-d| - |1-c| \bigr|
  < \varepsilon$,
  and the rest will follow.
  Possibly replacing $c$ with $1-c$, we may further assume that
  $|c| \leq |1-c| = 1$.
  If $|c| < 1$, just take for $d$ any positive element whose value is
  close enough to $|c|$, and if $|c| = 1$ choose $d$ so that $|d|$ is
  close enough to $1- \varepsilon/2$.
  This completes the proof.
\end{proof}

\begin{thm}
  The theory $RCMVF$ is dependent.
\end{thm}
\begin{proof}
  It is enough to show that every formula $\varphi(x,\bar y)$, where $x$ is a single variable, is dependent (this is shown in \cite{BenYaacov:RandomVC} along the lines of the proof for classical logic in \cite{Poizat:Cours}; a simplified argument appears in Adler \cite{Adler:IntroductionToDependent}, and it translates quite effortlessly to continuous logic).
  It is therefore enough to show that if $(\bar b_n)_n$ is an indiscernible sequence then $\bigl( \varphi(a,\bar b_n) \bigr)_n$ converges for every $a$.
  By quantifier elimination, we may assume that $\varphi$ is an atomic $\cL_{o\bP^1}$-formula, namely of the form $\| P(x,\bar y) \|$ or $\< P(x,\bar y) \>$.
  Since the type $p = \tp(\bar b_n)$ is constant, and since every field element which is algebraic over $\bar b_n$ is definable over $\bar b_n$ (because of the linear ordering), we may express $\| P(x,\bar b_n) \|$ and $\< P(x,\bar b_n) \>$ as continuous combinations of things of the form $|x-f(\bar b_n)|$ and $\< x-f(\bar b_n) \>$, where $f$ stands for a partial $\emptyset$-definable function whose domain contains $p$ (as in the proof of the previous theorem).
  For each such function, the sequence $\bigl( f(\bar b_n) \bigr)_n$
  is indiscernible as well, so in particular monotone, and it follows that $| a- f(\bar b_n) |$ and $\< a - f(\bar b_n) \>$ converge.
  This completes the proof.
\end{proof}

Alternatively, we may define $\cL_{o\bP}$ to consist of $\cL_\bP$
augmented with one predicate symbol $\< \cdot \>$ for each sort
$\bP^n$, $n \geq 1$, interpreted in an ordered valued field by
\begin{gather*}
  \< [a_0:\ldots:a_n] \>
  =
  \begin{cases}
    0 & a_0a_1 \geq 0, \\
    |a_0|\wedge |a_1| & \text{otherwise}.
  \end{cases}
\end{gather*}
We observe that this does not depend on the choice of
representatives (as long as $\bigvee |a_i| = 1$, as usual)
and this is compatible with the interpretation of $\< x \>$ on $\bP^1$
we introduced earlier.
One can extend \autoref{thm:ProjSpaceInterpretation}, showing that for an
ordered valued field $K$, the $\cL_{o\bP^1}$-pre-structure $K\bP^1$
and the $\cL_{o\bP}$-pre-structure $K\bP$ are quantifier-free
biïnterpretable, and this uniformly in $K$.

\providecommand{\bysame}{\leavevmode\hbox to3em{\hrulefill}\thinspace}

\end{document}